\documentclass[10pt] 
{amsart}
\usepackage{amsmath, amssymb, amsfonts, mathrsfs}
\usepackage{mathtools}

\makeatletter
\@namedef{subjclassname@2020}{\textup{2020} Mathematics Subject Classification}
\makeatother

\usepackage{epsfig}
\usepackage{color}
\usepackage{epstopdf}
\usepackage{graphicx}
\usepackage{srcltx}
\usepackage[mathscr]{eucal}
\usepackage{enumerate}
\usepackage{enumitem}
\usepackage{verbatim}
\usepackage[draft]{todonotes}
\usepackage{relsize}
\usepackage{kotex}
\usepackage
{hyperref}
\usepackage{cases} 
\usepackage{empheq}
\usepackage{subfigure}
\usepackage{tikz}
\usepackage{tikz-3dplot}
\usepackage{array}
\usepackage{tabularx}
\newcolumntype{C}[1]{>{\centering\arraybackslash}p{#1}}

\usetikzlibrary{arrows,automata}
\usepackage{varwidth}
\usepackage{tikz}
\usetikzlibrary{backgrounds,patterns,calc}
\usepackage{xparse}

\usepackage[normalem]{ulem}
\newcommand{\stkout}[1]{\ifmmode\text{\sout{\ensuremath{#1}}}\else\sout{#1}\fi}

\hypersetup{
	colorlinks=true,       
	linkcolor=blue,          
	citecolor=red,        
	filecolor=magenta,      
	urlcolor=cyan           
}

\usetikzlibrary{calc, arrows.meta, decorations.pathreplacing}

\newcommand{\Be}{\begin{equation}}
\newcommand{\Ee}{\end{equation}}
\newcommand{\Bea}{\begin{eqnarray}}
\newcommand{\Eea}{\end{eqnarray}}
\newcommand{\Bel}{\begin{align}}
\newcommand{\Eel}{\end{align}}
\newcommand{\Beas}{\begin{eqnarray*}}
	\newcommand{\Eeas}{\end{eqnarray*}}
\newcommand{\Benu}{\begin{enumerate}}
	\newcommand{\Eenu}{\end{enumerate}}
\newcommand{\Bi}{\begin{itemize}}
	\newcommand{\Ei}{\end{itemize}}

\def\abs#1{|#1|}

\numberwithin{equation}{section}
\newcommand{\supp} {\operatorname{supp}}

\theoremstyle{plain}
\newtheorem{theorem}{Theorem}[section]

\newtheorem{lemma}[theorem]{Lemma}
\newtheorem{proposition}[theorem]{Proposition}

\theoremstyle{remark}

\theoremstyle{definition}

\definecolor{ao}{rgb}{0, 0.5, 0}

\newcommand{\R}{\mathbb R}

\newcommand{\mm}{{M_\mu}}
\newcommand{\fa}{\mathfrak a}
\newcommand{\ts}{{2^{-j}t}}
\newcommand{\cI}{\mathcal I}
\newcommand{\tsk}{{2^kr}}

\newcommand{\etk}{\mathcal E^N}

\usepackage{tikz}
\usepackage{pgfplots}
\pgfplotsset{compat=1.13,
tick label style={color=white},
  label style={font=\small},
  legend style={font=\small}}

\usepackage{mathtools}

\title[On Rubio de Francia's maximal theorem]
{On Rubio de Francia's maximal theorem}
\author{Seheon Ham}
\author{Jiwon Kah}
\author{Sanghyuk Lee}
\author{Ji Li}

\address{Department of Mathematical Sciences and RIM, Seoul National University, Seoul 08826, Republic of Korea}
\email{seheonham@snu.ac.kr}
\email{wldnjskah@snu.ac.kr}
\email{shklee@snu.ac.kr}
\address{School of Mathematical and Physical Sciences, Macquarie University, NSW, 2109, Australia}
\email{ji.li@mq.edu.au}



\begin{document}

\maketitle

\begin{abstract}  In his influential 1986 paper, Rubio de Francia established   $L^p$ bounds for the maximal
function generated by dilations of measures $\mu$ whose Fourier transforms $\widehat{\mu}$ satisfy specific decay condition.
In the present work, we obtain results that complement his work in several directions.
In particular, we obtain restricted weak-type endpoint bound  on  the maximal function and
$L^p$--$L^q$ bounds on its local variant. We also investigate  how  Frostman’s growth condition  on the measure influences those maximal bounds.  
While a key feature of Rubio de Francia’s result is that $L^p$ boundedness is determined
solely by the decay order of $\widehat{\mu}$,  
we show that the Frostman condition plays a significant role when the growth order exceeds $d-1$ or when $L^p$--$L^q$ estimates are considered.

\end{abstract}

\section{Introduction}

Let $\mu$ be a compactly supported Borel measure in $\mathbb R^d$, $d\ge 1$.  For $t>0$,  let us denote  by $\mu_t$   its dilation given by 
\[  ( \mu_t,\phi) = \int    \phi(t y )\,  d\mu(y) \] 
for $\phi\in C_c(\mathbb R^d)$. Consider the maximal function that is generated by dilations of the measure  $\mu$: 
\[
M_\mu f(x) = \sup_{t>0} | f\ast\mu_t |. 
\] 


The following theorem concerning  $L^p$ boundedness for maximal operator   was proved by  Rubio de Francia in his influential paper  \cite{Ru86}.

\begin{theorem}[{\cite[Theorem A]{Ru86}}]\label{RF}
Suppose that $\mu$ is a compactly supported Borel measure and 
\begin{equation}\label{decay}
|\widehat\mu(\xi)|\le C |\xi|^{-a}   
\end{equation} 
with $a >\frac12$.
Then,  for all $ p >p_a := \frac{ 2a+1}{2a}$,   we have 
\begin{equation}\label{maximal}
 \| \mm f  \|_{L^p(\mathbb R^d)} \le C \|f\|_{L^p(\mathbb R^d)}. 
\end{equation}
\end{theorem}

The most typical example aligned with Theorem \ref{RF}  is the spherical maximal function  that is defined by 
\[ M_{sph} f= \sup_{t>0} |f\ast  \sigma_t|,\] 
where $\sigma$ is the normalized surface measure on $\mathbb S^{d-1}$.   
Theorem \ref{RF}  can be regarded as a natural generalization of 
Stein's celebrated spherical maximal theorem which tells that  the spherical maximal function $M_{sph}$   is bounded on $L^p$ 
 if and only if $p>p_{(d-1) /2}=d/(d-1)$.  This was shown by Stein \cite{St76} for $d\ge 3$, and later by  Bourgain \cite{B} for $d=2$. 
 
A result in a similar vein, relying on the decay condition \eqref{decay}, was
obtained earlier by Greenleaf~\cite[Theorem~2]{Gr} for measures $\mu$ supported on
smooth hypersurfaces. While Greenleaf’s approach relies on the specific geometric
structure of hypersurfaces, a notable feature of Theorem~\ref{RF} is that the
admissible range of $p$ is determined solely by the decay rate of the Fourier
transform of $\mu$. Consequently, Theorem~\ref{RF} remains valid for any compactly
supported measure satisfying \eqref{decay},  thereby ensuring its broad applicability across a wide range of settings.


\subsection{Endpoint estimate for $p=p_a$} 
In general, the range  $p> p_a$ is also sharp 
in that there are measures satisfying \eqref{decay} but  
the maximal estimate \eqref{maximal} fails if $p\le p_a$  for a certain $a$.     Indeed, the optimality of $p_a$ was shown  in  \cite[p. 397]{Ru86} when $a = k/2$ for integers $k$ satisfying $1 <  k \le d-1$ by making use of the usual example from \cite{St76}.

Regarding the critical exponent $p=d/(d-1)$ for the spherical maximal function,   Bourgain 
proved that   $M_{sph}$ is bounded from $L^{d/(d-1),1}$ 
to $L^{d/(d-1),\infty}$ when $d\ge 3$ \cite{B85}.  This endpoint result is also optimal from the viewpoint of  the second 
exponent of the Lorentz spaces. More precisely,    $M_{sph}$ cannot be  bounded from $L^{d/(d-1),r}$ 
to $L^{d/(d-1),\infty}$ for any $r>1$ (see  Proposition \ref{lorentz} below).  Nevertheless, such an estimate is not possible in the case $d=2$, as was shown in \cite{STW}.

It is natural to expect that an analogous endpoint result holds for the maximal operator $M_\mu$ satisfying \eqref{decay}. However, in contrast to the case of the spherical maximal function, to the best of the authors’ knowledge, no such endpoint estimate at $p = p_a$ has been established for $M_\mu$. 
  
  The first result in this paper is the following which establishes  the restricted weak-type estimate at the endpoint $p=p_a$.

\begin{theorem}\label{endp}
Suppose that $\mu$ is a compactly supported Borel measure and satisfying \eqref{decay} with $a >\frac12$. 
Then, we have 
\Be \label{endpt}  \| \mm f\|_{L^{p_a, \infty}(\R^d)} \le C \| f\|_{L^{p_a,1}(\R^d)}.  \Ee 
\end{theorem}

Moreover,  by combining  the argument in \cite[p.397]{Ru86} and Proposition \ref{lorentz-s} below, one can show that \eqref{endpt} cannot be extended to any larger Lorentz spaces $L^{p_a,r}$ for $r> 1$ (see Section \ref{sss})  when $a = k/2$  for integers $k$ satisfying $1<  k \le d-1$.

\subsection{Frostman condition}
The decay condition \eqref{decay} is related to geometric properties of the support of $\mu$, such as curvature. 
For example, if $\mu$ is compactly supported in a $k$-dimensional 
smooth submanifold satisfying certain non-degenerate condition, then \eqref{decay} holds with $a = k/2$ (see the discussion around \eqref{strongcurvature}).
For a general probability measure $\mu$, \eqref{decay}  implies that Hausdorff dimension of the support of $\mu$ is at least $2a$. 
Conversely, if $\mu$ is compactly supported in a set $E$ with Hausdorff dimension $s$, the decay rate $a$ in \eqref{decay} cannot exceed $s/2$.  (For example, see \cite{Ma}.)

Recently, considerable effort has been devoted to extending the classical results on Fourier restriction estimates to general measures beyond smooth measures supported on submanifolds \cite{Moc, Mit, BS, HaLa, Ch14, Ch16, CS, LW, FHR1, FHR2} (see also \cite{LP09, LP11}). 
In these studies, in addition to the Fourier decay condition \eqref{decay}, one typically imposes the Frostman condition
\begin{equation}\label{dim}
0< \mu(B(x,r)) \le C r^b,
\end{equation}
for some $0<b\le d$, which  serves to quantify the geometric properties of the support of $\mu$. 
Here, $B(x,r)$ denotes the ball of radius $r$ centered at $x\in\mathbb R^d$.
By Frostman's lemma, the condition \eqref{dim} implies that the Hausdorff dimension of the support of $\mu$ is at least $b$.
However, in general, there is no direct implication between the analytic Fourier decay condition \eqref{decay} and the geometric size condition \eqref{dim}.

In particular, the $L^2$ Fourier restriction estimate
\begin{equation}\label{TS}
\| \widehat{g\:\!d\mu} \|_{L^q} \le C \|g\|_{L^2(d\mu)}
\end{equation}
has been studied extensively under the conditions \eqref{decay} and \eqref{dim} by several authors \cite{Moc, Mit, BS}. Within this framework, the sharp range of exponents $q$ for which \eqref{TS} holds has also been investigated \cite{HaLa, Ch16, FHR1, FHR2} (see also \cite{CS, SS18, LW}). Furthermore, the existence of arithmetic progressions in the supports of fractal measures has been explored under the same assumptions \eqref{decay} and \eqref{dim} \cite{LP09, Sh17, CLP, HLP}. A common feature of these results is their reliance on both \eqref{decay} and \eqref{dim}, with the final conclusions depending heavily on both parameters $a$ and $b$. More recently, Carnovale, Fraser, and Orellana \cite{CFO} demonstrated that the estimate \eqref{TS} can be further refined by incorporating additional information regarding the Fourier spectrum of the measure.

However, as already discussed above, the maximal estimate \eqref{maximal} in Theorem \ref{RF} is independent of the Frostman condition \eqref{dim}. More precisely, assuming the additional condition \eqref{dim} with some $b \le d-1$ does not affect the range of $L^p$ boundedness, as will be shown below. Interestingly, a transition occurs when the Frostman condition \eqref{dim} holds with $b > d-1$; in this regime, the range of $L^p$ boundedness begins to depend explicitly on the parameter $b$.

Our second result characterizes this dependence by quantifying how the admissible range for the maximal bound \eqref{maximal} can be extended under the assumption \eqref{dim} for $b > d-1$. This phenomenon stands in sharp contrast to the case $b \le d-1$, where the Frostman condition plays no role in determining the $L^p$ range.

\begin{theorem}\label{decay+dim}
Let $\mu$ be a compactly supported positive Borel measure in $\mathbb R^d$ satisfying \eqref{decay} with $a >1/2$ and \eqref{dim} with $d-1< b \le d$.
Then,   we have  the estimate \eqref{maximal}  for 
\Be
\label{pab}  p >  p_{(a,b)}:= \frac{ 2(d-b) + 2a -1}{d-b+2a-1}.
\Ee
Moreover, we have the endpoint estimate  
\Be\label{rsweak}  \| \mm f\|_{L^{p_{(a,b)}, \infty}(\R^d)} \le C \| f\|_{L^{p_{(a,b)},1}(\R^d)}.\Ee
\end{theorem}

The range $p>p_{(a,b)}$ is sharp in general, as can be seen from specific examples showing that the estimate \eqref{maximal} fails when $p\le p_{(a,b)}$ (see Section~\ref{sharpness}).
Note that $p_{(a,b)}<p_a$ if and only if $b>d-1$.
Thus, Theorem~\ref{decay+dim} yields a wider range of boundedness than that of  Theorem~\ref{endp} when $b>d-1$.

\subsection{Local maximal function}   
Let $I = [1,2]$. 
We now consider the local maximal operators $M_\mu^{loc}$ defined  by 
\begin{equation*}
M_\mu^{loc} f(x) =\sup_{t\in I} |f\ast \mu_t(x)|. 
\end{equation*}
Due to scaling invariance,   the global maximal operator $M_\mu$ can be bounded only from $L^p$ to itself. However, 
it is known that  the local maximal operator $M_\mu^{loc}$ has $L^p$-improving property for some specific measures, for example, when $\mu=\sigma$.  That is to say, the estimate 
\begin{equation}\label{pqmax}
\| M_\mu^{loc} f   \|_{L^q(\R^d)}  \le C \|f\|_{L^p(\R^d)}
\end{equation}
holds for some $q >p$. In particular,     the $L^p$-improving property   for the local spherical  maximal function $M_{sph}^{loc}:=M_\sigma^{loc}$    is almost completely understood except some endpoint cases  
\cite{Schlag, SS, L} (see Section \ref{ext} for details).  

Likewise, one can expect that  $M_\mu^{loc}$ also has $L^p$-improving property for general measures satisfying \eqref{decay} and \eqref{dim}. Indeed, under these assumptions,  we show that   $M_\mu^{loc}$ is bounded from $L^{p_{(a,b)},1}$ to $L^{p'_{(a,b)},\infty}$, where  $p'$  denotes the H\"older conjugate exponent of $p$. (See Theorem \ref{thm:pq}).   Additionally,  assuming a kind of dispersive estimate for the convolution operator $f\to \mu_t\ast \overline \mu_s\ast f$, we also show that  the range of $p,q$ for which \eqref{pqmax} holds can be further extended (see Theorem \ref{pqr}). 
 
\subsubsection*{Organization}    In Section \ref{sec:endp}, we prove Theorem \ref{endp}. Section \ref{sec3} is devoted to the proof of Theorem \ref{decay+dim}, together with a discussion of its sharpness. Finally, in Section \ref{sec4}, we study the $L^p$-improving property of the local maximal operator $M_\mu^{\mathrm{loc}}$.
We also take this opportunity to correct an error in the statement of
\cite[Theorem~1.4]{L}, which concerns endpoint mapping properties of
$M_{sph}^{loc}$ for $d\ge3$; see Theorem~\ref{leelee} for details.

\section{Proof of Theorem \ref{endp}}\label{sec:endp} 
In this section we prove Theorem \ref{endp}. We begin by decomposing the multiplier $m(\xi)$ in the same manner  as in \cite{Ru86}.

Let $\beta\in C_c^\infty((1/2,2))$  such that 
 $\sum_{j=-\infty}^{\infty}\beta(2^{j}t)=1$ for $t>0$.  We set
\begin{align}\label{varphi0}  
\varphi_0(\xi) = \sum_{j\le 0} \beta(2^{-j}|\xi|),
\end{align}
and 
\begin{align}\label{varphij}  
  \varphi_j (\xi) = \beta(2^{-j}|\xi|), \quad j\ge 1.
\end{align}

\subsection{Decomposition of   $\mu$ and $L^2$ estimate}\label{sec:L2}  

Since $\sum_{j=0}^\infty \varphi_j(\xi)=1$ for all $\xi\in\mathbb{R}^d\setminus\{0\}$,  we
 may write
\[
\widehat \mu=\sum_{j=0}^\infty m_j,  
\]
where 
\[ m_j=\widehat \mu\,\varphi_j.\]

 Let $T_j^t $ denote the multiplier operator given by 
\[
\widehat{T_j^t  f}(\xi) = m_j(t\xi)\,\widehat f(\xi).
\]
Consequently, we have
\[
\mm f(x)\le \sum_{j=0}^\infty T_j^* f(x),
\]
where 
\begin{equation}
\label{eq:tjst}
T_j^* f(x) = \sup_{t>0} |T_j^t  f(x)|,
\end{equation}

Observe that $T_0^*$ is bounded on $L^p(\mathbb{R}^d)$ for all
$1<p\le \infty$. Indeed, the kernel
\[
K_0(x)=\mathcal{F}^{-1}(m_0)(x)
\]
is in  $\mathcal S(\mathbb R^d)$. Thus,  $T_0^* f$ is bounded above by a constant times the Hardy--Littlewood maximal function, which 
is  bounded on $L^p$ for $1<p\le \infty$. 
Thus,  in order to show the endpoint estimate \eqref{endpt}, it suffices to prove  
\Be \label{rweak0}     \Big\| \sum_{j\ge 1} T_j^\ast  f\Big\|_{L^{p_a, \infty}(\R^d)} \le C \|f\|_{L^{p_a,1}(\R^d)}.\Ee

For the purpose \eqref{rweak0},  we make use of the following, which is known as  Bourgain's summation trick \cite{B85} (see, also,  {\cite[Lemma 2.6]{L}}).

\begin{lemma}
\label{B-Sum-Tri} 
Let  $1\leq p_0,\, p_1\le \infty$,  and  $1\leq q_0,\, q_1\le\infty$. Suppose that $\{T_j\}_{j=-\infty}^\infty$ is a sequence of linear (or sublinear) operators such that
\begin{equation*}
\|T_j f\|_{L^{q_\ell}(\R^d)}\leq M_l 2^{j (-1)^\ell \varepsilon_\ell } \|f\|_{L^{p_\ell}(\R^d)},  \quad \ell=0, 1
\end{equation*}
for some  $\varepsilon_0$, $\varepsilon_1>0$. 
Then, we have 
\begin{equation*}
\Big\|\sum_jT_j f  \Big\|_{L^{q,\infty}(\R^d)}\leq CM_1^{\theta}M_2^{1-\theta} |f\|_{L^{p,1}(\R^d)},
\end{equation*}
where $\theta=\varepsilon_1/(\varepsilon_0+\varepsilon_1)$, $1/q=\theta/q_0+(1-\theta)/q_1$, and $1/p=\theta/p_0+(1-\theta)/p_1$.
\end{lemma}

Thanks to Lemma \ref{B-Sum-Tri},  the desired estimate follows once we have 
\begin{equation}  
\label{eq:tj} 
\| T_j^* f\|_{L^p(\R^d)}\le  C  2^{(2a+1)(\frac1p-\frac{2a}{2a+1})j} \|f\|_{L^p(\R^d)} , \quad  j\ge1
\end{equation}
for $1<p \le 2$.  Indeed, taking $p_0$ and $p_1$  such that $2> p_0> \frac{2a+1}{2a}> p_1>1$ with $\epsilon_j=(-1)^j(2a+1)(\frac{2a+1}{2a}-\frac1{p_j})$, $j=0, 1$, 
by Lemma \ref{B-Sum-Tri} we obtain \eqref{rweak0}. 

In order to show \eqref{eq:tj}, we first recall the next lemma that is already proved in Rubio de Francia's paper (\cite[Lemma 1]{Ru86}), which is 
an easy consequence of \eqref{decay} and Plancherel's theorem. 

\begin{lemma}\label{lem:L2-Tj}
For every $j\ge1$, we have
\[
\|T_j^* f\|_{L^2(\mathbb{R}^d)}
   \le C\,2^{j(\frac12-a)} \|f\|_{L^2(\mathbb{R}^d)}.
\]
\end{lemma}

Thus,  by interpolation, the proof of \eqref{eq:tj} is reduced to  showing the following estimate from the standard Hardy space $H^1(\R^d)$ to $L^1(\R^d)$. 

\begin{proposition}\label{lem:H1-tj}
For every $j\ge1$, we have 
\begin{equation}\label{eq:H1-goal}
    \|T_j^* f\|_{L^1(\mathbb{R}^d)} \le C\,2^j \|f\|_{H^1(\mathbb R^d)}. 
\end{equation}
\end{proposition}

The estimate \eqref{eq:H1-goal} improves upon the bound 
\[ \|T_j^* f\|_{L^1(\mathbb{R}^d)} \le  C j 2^{j} \|f\|_{H^1(\mathbb{R}^d)}\] 
 obtained in \cite{Ru86} by eliminating the logarithmic loss.
This  bound was  proved via the Calder\'on--Zygmund theory,
more precisely, by controlling $\|T_j^*\|_{H^1(\mathbb{R}^d) \to L^1(\mathbb{R}^d)}$ mainly  through a generic
condition on the modulus of continuity.

The rest of this section is devoted to the proof of Proposition  \ref{lem:H1-tj}.

\subsection{Estimate with a $H^1(\mathbb{R}^d)$ atom}
In order to prove \eqref{eq:H1-goal},  by the standard Hardy space theory,  it is enough to show that 
\begin{equation}\label{eq:atom-goal}
    \|T_j^* \fa\|_{L^1(\mathbb{R}^d)} \le C\,2^j
\end{equation}
 for any $H^1$-atom $\fa$ with a positive constant $C$,  independent of $j$ and $\fa$. 

Recall that   an $H^1(\mathbb{R}^d)$ atom $\fa$ satisfies the following: 
\begin{align*}
\tag{A1} 
\mathrm{supp}(\fa)&\subset B, 
\label{a1}
\\ 
\tag{A2} 
 \|\fa\|_{L^\infty(\mathbb{R}^d)}&\le |B|^{-1}, 
\label{a2}
\\ 
\tag{A3}  
\int_{\R^n} \fa(x)\,dx &= 0, 
\label{a3}
\end{align*}
where $B$ is a ball in $\mathbb R^d$. 
From \eqref{a1} and \eqref{a2} we also have the basic size bound
\begin{equation}\label{eq:atom-L1}
\|\fa\|_{L^1(\mathbb{R}^d)}\le 1.
\end{equation}
Indeed, $\|\fa\|_{L^1(\mathbb{R}^d)}\le |B|\,\|\fa\|_{L^\infty(\mathbb{R}^d)} \le |B|\,|B|^{-1}=1$.

\subsection{Preliminary decomposition}
Since $\mu$ is compactly supported and the condition \eqref{decay} is invariant under translation and dilation (up to a change in the constant $C$), we may assume without loss of generality that
\begin{equation} \label{musupp} \supp \mu \subset B(2^{-1}e_1, 2^{-2}) \subset B(0,1). \end{equation}
Moreover, since $T_j^*$ is translation invariant,  to show the estimate \eqref{eq:atom-goal}  we may translate coordinates and assume that 
the ball $B$ is centered at the origin; thus 
\Be\label{asupp} \mathrm{supp}\,\fa\subset B=B(0,r).\Ee

To prove \eqref{eq:atom-goal}, we break 
\begin{align*}
    \|T_j^* \fa\|_{L^1(\mathbb{R}^d)}     &= I_{near} + I_{far},
\end{align*}
where 
\begin{align}
\label{eq:neqr}
I_{\text{near}} &:=\int_{ |x|\leq4r } T_j^* \fa(x)dx,
\\
\label{eq:far} 
I_{\text{far}}&:= \int_{ |x|>4r } T_j^* \fa(x)dx .
\end{align}

\subsection{Estimate for $I_{\mathrm{near}}$} We first estimate the maximal operator near the support of the atom, 
specifically over the region $\{|x| \le 4r\}$. Recalling \eqref{eq:neqr}, we use the Cauchy--Schwarz inequality and   Lemma~\ref{lem:L2-Tj}  to obtain 
\begin{align*}
I_{\mathrm{near}}
\le \bigl|B(0,4r)\bigr|^{1/2}\,\|T_j^* \fa\|_{L^2(\R^d)}
\lesssim  r^{d/2}\,2^{j(1/2-a)}\|\fa\|_{L^2(\mathbb{R}^d)}. 
 \end{align*}
Since $\|\fa\|_{L^2(\mathbb{R}^d)}\le \|\fa\|_{L^\infty(\mathbb{R}^d)} |B(0,r)|^{1/2}\lesssim r^{-d/2}$, 
\begin{equation}\label{eq:near-final}
I_{\mathrm{near}}
\lesssim  2^{j(1/2-a)}.
\end{equation}
In particular, since $a>1/2$, the right-hand side is $\lesssim 1$ uniformly in $j$.

\subsection{Estimate for $I_{\mathrm{far}}$} Estimate for $I_{\mathrm{far}}$ is more involved.
To this end, we make use of a locally constant property of the associated kernels.

We begin by noting that
\[
\widehat{T_j^t f}(\xi)= \widehat \mu(t\xi)\,\varphi_j(t\xi)\, \widehat f(\xi).
\]
Thus, we have
\[
T_j^t f(x) = K_j^t *f(x),
\]
where
\begin{equation}\label{eq:Kjt-explicit}
\begin{aligned}
K_j^t (x)
&= 2^{dj} t^{-d} \int_{\R^d} {\beta}^\vee \left(2^j\Big(\frac{x}{t} - u\Big)\right) \, d\mu(u) .
\end{aligned}
\end{equation}
This  will be often used  throughout  this section  to obtain bounds on the kernel $K_j^t $.

Recalling \eqref{eq:tjst} and \eqref{eq:far}, 
we split 
\[
I_{\mathrm{far}}
\le I_{\mathrm{far}}^{(1)} + I_{\mathrm{far}}^{(2)},
\]
where
\begin{align}
I_{\mathrm{far}}^{(1)}&:=\int_{|x|>4r}\sup_{ 0<t\leq 2^jr}\abs{(K_j^t *\fa)(x)}\,dx,
\label{eq:f2}
\\
I_{\mathrm{far}}^{(2)}&:=\int_{|x|>4r}\sup_{t> 2^jr}\abs{(K_j^t *\fa)(x)}\,dx.
\label{eq:f3}
\end{align}

\subsection{Estimate for $I_{\mathrm{far}}^{(1)}$}\label{subsec:far-intermediate} 
We first handle  $I_{\mathrm{far}}^{(1)}$, for which we obtain the estimate 
\Be\label{ff2}
I_{\mathrm{far}}^{(1)} \le \ C 2^j\,\|\mu\|. 
\Ee 
Here $\|\mu\|$ denotes the total mass of  $\mu$.

Recalling \eqref{eq:Kjt-explicit}, from rapid decay of ${\beta}^\vee$ we note
\[
|K_j^t (x)|\lesssim \int_{\R^d} \etk_\ts (x-tu)\,d|\mu|(u)
\]
for any $N$, where
\Be
\label{ekt}
 \etk_\rho(x):= \rho^{-d} ( 1+ \rho^{-1}|x|)^{-N}.
\Ee

 Changing the order of the integrations gives
\begin{equation}\label{eq:mid-rewrite}
| K_j^t *\fa(x)| 
\lesssim \int_{\R^d}\bigl( \etk_\ts*|\fa|\bigr)(x-tu)\,d|\mu|(u),
\end{equation}
where $|\mu|$ denotes the total variation of $\mu$. 
Thus, we have 
\begin{equation}\label{eq:mid-sup-inside}
\sup_{0<t\le 2^j r}|(K_j^t *\fa)(x)|
\le \int_{\R^d}\sup_{0<t\le 2^j r}\big( \etk_\ts*|\fa|\bigr)(x-tu)\,d|\mu|(u).
\end{equation}

\bigskip

Fix a large  integer $N>2d+1$. We first claim that
\begin{equation}\label{eq:gt-majorized}
\sup_{0<t\le 2^j r} \etk_\ts *|\fa| (x)  \lesssim   \etk_r (x), 
\end{equation}
 where the implicit constant depends only on $N,d$ but is independent of $r$ and $x$.  

\begin{proof}[Proof of  \eqref{eq:gt-majorized}.] 
Note $\|\fa\|_{L^\infty(\mathbb{R}^d)}\lesssim r^{-d}$. Since  $
\big( \etk_\ts*|\fa|\bigr)(x)  
\lesssim r^{-d}$,   to show \eqref{eq:gt-majorized} we may assume that 
\[  |x|\ge 2r.\]
Consequently, we need only  to show that 
\Be 
\label{easy}
\bigl( \etk_\ts*|\fa|\bigr)(x)   \lesssim r^{-d} \left( \frac{|x|}{ r}\right)^{-N}. 
\Ee

Since $0< \ts\le r$ and $N>d$, it follows that 
\begin{align*}
\etk_\ts (x-y) \le {(\ts)}^{-d}\Bigl(\frac{|x-y|}{\ts}\Bigr)^{-N}
\le r^{\,N-d}|x-y|^{-N}.
\end{align*}
The fact that $\supp \fa\subset B(0,r)$ and the above inequality yield 
\begin{align*}\label{eq:gt-majorized case2}
\big( \etk_\ts*|\fa|\bigr)(x)   &\le  \|\fa\|_{L^\infty(\mathbb{R}^d)} \int_{|y|\le r}   \etk_\ts (x-y) \,dy\\
&\lesssim r^{-d}\int_{|y|\le r} r^{\,N-d}|x-y|^{-N}\,dy.
\end{align*}
Recalling $|x|\ge 2r$, we note that $|x-y|\ge |x|/2$ for $|y|\le r$. Thus,  
\begin{align*}
\big( \etk_\ts*|\fa|\bigr)(x)   &\lesssim r^{-d}\,r^{\,N-d}\,|x|^{-N}\,|B(0,r)|, 
\end{align*}
which  gives  \eqref{easy}. 
 \end{proof}

Combining 
\eqref{eq:mid-sup-inside}  and \eqref{eq:gt-majorized} gives
\begin{align*}
I_{\mathrm{far}}^{(1)} 
&\lesssim \int_{\R^d}\int_{|x|>4r}\sup_{0<t\le 2^j r} \etk_r  (x-tu)\,dx\,d|\mu|(u).
\end{align*}
Thus, recalling  \eqref{musupp}, for \eqref{ff2} we need only  to show 
\[ 
\int_{|x|>4r}\sup_{0<t\le 2^j r} \etk_r  (x-tu)\,dx 
\lesssim 2^j
\]
for $1/4 \le |u| \le 3/4$,  equivalently, 
\begin{equation}\label{eq:Fu-L1}
\int_{|x|>4r}\sup_{0<t\le 2^j r|u|} \etk_r  (x-t\omega)\,dx 
\lesssim 2^j,  \quad   \omega:=\frac{u}{|u|}. 
\end{equation}

Consequently, we are led to consider  domination of the  supremum of $\etk_r (x-\cdot)  $ along the line segment 
\[ L_\omega:=\{t\omega:\ 0<t\le 2^j r |u|\}.\]
Setting  $t_k=kr,$ $k=0, \dots,    M= \lceil 2^j r \rceil $, we have  
\[  L_\omega  \subset   \bigcup_{k=0}^{M-1}   L_k :=\bigcup_{k=0}^{M-1} [t_k \omega, t_{k+1} \omega].\]
  Note that 
\begin{align*}
\sup_{0<t\le 2^j r|u|} \etk_r  (x-t\omega) = \sup_{h\in L_\omega} \etk_r  (x-h) 
&\le \sum_{k=0, \dots, M-1} \sup_{h\in L_k} \etk_r  (x-h) .
\end{align*}
Therefore,  we have
\Be 
\label{tem0}
\sup_{0<t\le 2^j r|u|} \etk_r  (x-t\omega) \le  \sum_{k=0, \dots, M-1}  \Big( \sup_{|z|\le r}\etk_r  (x-t_k\omega-z)  \Big).
\Ee

We now make use of the following elementary lemma which exploits  locally constant  property of $\etk_r$ at scale $r$. 

\begin{lemma}\label{l-constant} Let $\rho>0$. There exists a constant $C=C(N)$ such that
\begin{equation}\label{eq:slow-varying} 
C^{-1} \etk_\rho(y) \le \sup_{|h|\le \rho} \etk_\rho(y-h) \le C \etk_\rho(y)  
\end{equation}
for all $y\in\R^d.$
\end{lemma}

\begin{proof}
When $|h|\le \rho$, note  that 
\[
2^{-1} (
 \rho+ {|y|})\le  \rho+ {|y-h|}   \le 2 (
 \rho+ {|y|})
\] 
for all $y\in \mathbb R^n$.  Thus, it follows that 
\[
2^{-N}  \Bigl(1+\frac{|y|}{\rho}\Bigr)^{-N} \le \Bigl(1+\frac{|y-h|}{\rho}\Bigr)^{-N}\le 2^N  \Bigl(1+\frac{|y|}{\rho}\Bigr)^{-N}.
\]
Then, recalling \eqref{ekt}, we get the desired inequality \eqref{eq:slow-varying}. 
\end{proof}

Applying Lemma  \ref{l-constant} to the right hand side of \eqref{tem0}, we obtain 
\[
\sup_{0<t\le 2^j r|u|} \etk_r  (x-t\omega)\le  C \sum_{k=0}^{M-1} \etk_r  (x-t_k \omega).
\]
Integrating over the region $|x|\ge 4r$ yields 
\begin{align*}
\int_{|x|>4r}\sup_{0<t\le 2^j r|u|} \etk_r  (x-t\omega)\,dx
&\le C \sum_{k=0}^{M-1}  \int_{\mathbb R^d}\etk_r  (x-t_k \omega)\,dx.
\end{align*}
Therefore, we have 
\begin{align*}
\int_{|x|>4r}\sup_{0<t\le 2^j r|u|} \etk_r  (x-t\omega)\,dx &\le CM\,\|\etk_r  \|_{L^1(\R^d)}. 
\end{align*}
Note that $M \approx \frac{2^jr|u|}{r} =2^j|u|$ and $\|\etk_r  \|_{L^1(\R^d)} \le C$ with a constant  $C$ independent of $r$. Therefore,  we obtain \eqref{eq:Fu-L1} since $|u|\le 1$ and $j\ge 1$.
This completes the proof \eqref{ff2}.

\subsection{Estimate for $I_{\mathrm{far}}^{(2)}$}
We now consider  $I_{\mathrm{far}}^{(2)}$ and prove 
\Be\label{ff3}
I_{\mathrm{far}}^{(2)} \le \ C2^j\,\|\mu\|,
\Ee 
which,  combined  with \eqref{ff2}, gives  $I
_{\mathrm{far}} \le \ C2^j\,\|\mu\|$.  Consequently, it proves \eqref{eq:atom-goal}. 
Thus, it only remains to show \eqref{ff3} to complete the proof.

 Using the cancellation condition \eqref{a3} of the atom $\fa$:
$\int \fa(y)\,dy=0,$
we have
\begin{align}
(K_j^t *\fa)(x) = \int_{|y|\le r}\bigl(K_j^t (x-y)-K_j^t (x)\bigr)\fa(y)\,dy.
\label{eq:cancellation-step}
\end{align}
The fundamental theorem of calculus in the direction $y$ gives
\[ 
K_j^t (x-y)-K_j^t (x)
= \int_0^1 y\cdot (\nabla K_j^t )(x-sy)\,ds.
\] 
Inserting this into \eqref{eq:cancellation-step} and then  taking absolute values  yield 
\[ 
\sup_{t>2^jr}\abs{(K_j^t *\fa)(x)}
\le \int_{|y|\le r} |y|\,|\fa(y)|\int_0^1 \sup_{t>2^jr}\abs{\nabla K_j^t (x-sy)}\,ds\,dy.
\]

Consequently, recalling \eqref{eq:f3} and  integrating over $\{|x|>4r\}$, we have
\begin{align}
I_{\mathrm{far}}^{(2)}
&= \int_{|y|\le r} |y|\,|\fa(y)|\int_0^1 \int_{|x|>4r}\sup_{t>2^jr}\abs{\nabla K_j^t (x-sy)}\,dx\,ds\,dy.
\label{eq:Tonelli}
\end{align}

Since  $|x|>4r$, $|y|\le r$, and $s\in[0,1]$, we have 
$|x-sy|\ge 3r. $  Thus, the change of variables $z=x-sy$ maps $\{|x|>4r\}$ into $\{|z|>3r\}$. Hence, 
\[
\int_{|x|>4r}\sup_{t>2^jr}\abs{\nabla K_j^t (x-sy)}\,dx
\le \int_{|z|>3r}\sup_{t>2^jr}\abs{\nabla K_j^t (z)}\,dz.
\]
Combining this and  \eqref{eq:Tonelli}, we obtain 
\begin{equation*}
I_{\mathrm{far}}^{(2)}
\le \left(\int_{|y|\le r}|y|\,|\fa(y)|\,dy\right)\,
\left(\int_{|z|>3r}\sup_{t>2^jr}\abs{\nabla K_j^t (z)}\,dz\right).
\end{equation*}
Using \eqref{eq:atom-L1}, we have  $
\int_{|y|\le r}|y|\,|\fa(y)|\,dy \le r.$
Therefore, the desired estimate \eqref{ff3} follows if we show 
\begin{equation}\label{eq:reduce-to-J}
J_r:=\int_{|x|>3r}\sup_{t>2^jr}\abs{\nabla K_j^t (x)}\,dz \le  C\frac{2^j}{r} \|\mu\|.
\end{equation}

For $\rho>0$,  set 
\Be\label{etil}   \tilde {\mathcal E}^N_\rho  = \rho^{-1}  \etk_\rho .\Ee
Recalling \eqref{eq:Kjt-explicit}, we note that 
\[| \nabla K_j^t (x)|\lesssim  \int_{\R^d} \tilde {\mathcal E}^N_\ts (x-tu)\,d|\mu|(u).  \]
Thus,  from  \eqref{eq:reduce-to-J}   we have
\[
J_r 
\le 
 \int_{\R^d} J_r(u)  d|\mu|(u),
\] 
where 
\[  J_r(u)= \int_{\R^d}  \sup_{2^j r  \le t}  \tilde {\mathcal E}^N_\ts (x-tu)\, dz.\]
Therefore,   the estimate  \eqref{eq:reduce-to-J} follows if we prove 
\Be 
\label{ju}
J_r(u) \le  C\frac{2^j}{r}, \quad j\ge 1. 
\Ee

In order to exploit the  decay in $t$, we make dyadic decomposition of the set  $\{t: 2^j r  \le t\}$  such that  
\[   \{t: 2^j r  \le t\}=\bigcup_{k=0}^\infty  \mathcal I_k :=  \bigcup_{k=0}^\infty   [ 2^k 2^j  r,     2^{k+1}  2^j r ].\]
Consequently, 
\[  \sup_{2^j r \le t} \tilde {\mathcal E}^N_\ts (x-tu)  \le \sum_{k=0}^\infty   \sup_{t\in \cI_k } \tilde {\mathcal E}^N_\ts (x-tu)  .\]
To show \eqref{ju}, we handle  $ \sup_{t\in \cI_k } \tilde {\mathcal E}^N_\ts  (x-tu)$, separately. Indeed, \eqref{ju} follows if we prove 
 \[
\int_{\mathbb R^d}  \sup_{t\in \cI_k } \tilde {\mathcal E}^N_\ts  (x-tu)\, dx  \le C\frac{2^j}{ 2^k r}. 
\]
Summation over $k$  gives \eqref{ju}.

Since $2^{-j} t\in [ 2^kr,  2^{k+1}r]$ for $t\in \cI_k$, it is easy to see $ \tilde {\mathcal E}^N_\ts  \lesssim    \tilde {\mathcal E}^N_{2^kr} .$
Moreover, recalling  \eqref{musupp}, we may assume that  $1/4 \le |u| \le 3/4$. 
Thus, the matter is reduced to showing 
\Be 
 \label{juk}
\int_{\mathbb R^d}  \sup_{t\in \cI_k } \tilde {\mathcal E}^N_\tsk  (x-tu)\, dx  \le C\frac{2^j}{2^k r}
 \Ee
 for $1/4 \le |u| \le 3/4$. 
To prove this, we follow the similar strategy used before to prove the estimate \eqref{ff2}.

Since $1/4 \le |u| \le 3/4$, one can easily see that $\tilde {\mathcal E}^N_\tsk(|u|x)\sim \tilde {\mathcal E}^N_\tsk(x)$. Thus,  by changing variables 
$x\to |u|x$, we may assume $|u|=1$. From Lemma \ref{l-constant},  note that   $\tilde {\mathcal E}^N_\tsk$ has a locally constant  property at scale $\tsk$.
Thus, we cover  
  the interval  $\cI_k $ with essentially disjoint  intervals  $I_\ell$ of length 
  \[ \delta:=2^kr\]  
  such that 
\[   \cI_k \subset \bigcup_{\ell=0}^{M-1} I_\ell:=\bigcup_{\ell=0}^{M-1} [t_\ell, t_{\ell+1}  ],   \quad \cI_k\cap I_\ell\neq \emptyset    .\]

Note that $|t_\ell u-  t_{\ell+1} u|=  2^k r$. 
Using Lemma \ref{l-constant}, as before, we obtain 
\[ \sup_{t\in \cI_k } \tilde {\mathcal E}^N_\tsk  (x-tu)\lesssim \sum_{\ell=0}^{M-1}  \tilde {\mathcal E}^N_\tsk  (x-t_\ell u).\]
Recalling \eqref{etil}, we note that $\|\tilde {\mathcal E}^N_\rho\|_{L^1(\mathbb{R}^d)}\lesssim \rho^{-1}$. Thus, we have 
\begin{align*}
\int_{\mathbb R^d}  \sup_{t\in \cI_k } \tilde {\mathcal E}^N_\tsk  (x-tu)\, dx   \le \sum_{\ell=0}^{M-1}  \| \tilde {\mathcal E}^N_\tsk  (\cdot-t_\ell u)\|_{L^1(\mathbb{R}^d)} \lesssim  \frac M {\tsk}. 
\end{align*}
Sine the length of $\cI_k$ is $2^{k}  2^j r$,  $M\le   2^j + 2$.   
Therefore, we obtain \eqref{juk}, which completes the proof of Theorem \ref{endp}.

\section{Proof of Theorem \ref{decay+dim}}
\label{sec3}

In this section, we prove Theorem \ref{decay+dim}. As noted previously, conditions \eqref{decay} and \eqref{dim} are invariant under translation and dilation of the measure $\mu$, up to a change in the constant $C$. Since $\mu$ is compactly supported, as before we may assume without loss of generality, via suitable translation and dilation, that   
\[ \supp \mu\subset B(0,1)\] (cf. \eqref{musupp}).

\begin{lemma}\label{conv}  
Let $R \ge 1$ and $N> d+1 $. Recalling \eqref{ekt}, let $\phi_R = \etk_\rho$  with $\rho=R^{-1}$. Suppose  $\mu$ is a Borel measure  satisfying  \eqref{dim}  with 
$\supp \mu\subset B(0,1)$. 
Then, for $t\in I:=[1,2]$, we have
\begin{equation*}
 | \phi_R \ast \mu_t (x)| \le C R^{d-b} (1 + |x|)^{-N+d+1} .
\end{equation*}
\end{lemma}

\begin{proof} We begin by noting that   
\begin{align*}
| \phi_R \ast \mu_t(x) |   \le  R^d \int (1 + R|x-ty|)^{-N} d\mu(y)  .
\end{align*}
Since $R\ge 1$ and 
$\supp \mu\subset B(0,1)$,  it follows  that  
$ (1+ R|x-ty|)^{-1} \le C (1 + |x|)^{-1}$ for $t\in I$ if $|x|\ge 3$. 
Thus, it is sufficient to show that 
\[  \int (1 + R|t^{-1}x-y|)^{-(d+1)} d\mu(y) \le CR^{-b} \]
because $t\in [1,2]$. By dyadic decomposition, we note that 
\[  (1 + R|t^{-1}x-y|)^{-(d+1)} \lesssim   \sum_{k=0}^\infty   2^{-(d+1)k}  \chi_{B(t^{-1}x, 2^kR^{-1})}(y).   \] 
Thus,  we have
\[ \int (1 + R|t^{-1}x-y|)^{-(d+1)} d\mu(y)  \lesssim  \sum_{k\ge 0}   2^{- (d+1)k} \mu (B(t^{-1}x, 2^k R^{-1})). \] 

Therefore, using \eqref{dim}, we have
\begin{align*}
\int (1 + R|t^{-1}x-y|)^{-(d+1)} d\mu(y) &\lesssim \sum_{k\ge 0}   2^{- (d+1)k} (2^k R^{-1})^b\le C  R^{-b}
.\end{align*}
This completes the proof.
\end{proof}

\subsection{Decomposition via Littlewood--Paley decomposition} 
Let $P_j$ denote  the standard Littlewood--Paley projection operator given by \[ \widehat{P_j f}(\xi) = \beta(2^{-j}|\xi|)\widehat f(\xi).\]
In particular, for each $k\in \mathbb Z$, we have
\Be
\label{k-decom}
 f(x) =  P_{\le k} f(x) +  \sum_{j \ge1} P_{k+j} f(x), 
\Ee
where $P_{\le k} = \sum_{j \le k} P_j$.

 Note that
\[ 
M_\mu f=\sup_{t>0}  |f\ast \mu_t|= \sup_{-\infty < k <\infty} \sup_{2^{-k}\le t< 2^{-k+1}}  |f\ast \mu_t|. 
\]
Using \eqref{k-decom},  we have
\[
\sup_{2^{-k}\le t< 2^{-k+1}}  |f\ast \mu_t|\le   \sup_{2^{-k}\le t< 2^{-k+1}}  |P_{\le k} f\ast \mu_t|  +  \sum_{j=1} 
^{\infty} \sup_{2^{-k}\le t< 2^{-k+1}}  |P_{k+j}f\ast \mu_t|.
\]
Therefore, taking supremum over $k$, we have 
\[ 
M_\mu f\le   \sum_{j=0} 
^{\infty} M_j f, 
\]
where 
\begin{align}
   M_0 f&=   \sup_{-\infty < k <\infty}   \sup_{2^{-k}\le t< 2^{-k+1}}  |P_{\le k} f\ast \mu_t|,  \label{m0}
   \\
  M_j f&= \sup_{-\infty < k <\infty}   \sup_{2^{-k}\le t< 2^{-k+1}}   |P_{k+j}f\ast \mu_t|, \quad j\ge 1. 
  \label{mj}
  \end{align}

With a large integer $N> 2d+1$, for $\ell \in \mathbb Z$,  let us set 
\Be\label{tk}   \tilde K_\ell   (x)=\etk_{2^{-\ell}}. \Ee
For $t\in [2^{-k}, 2^{-k+1}] $ and $j\ge 0$,  we have
\Be
\label{kast}
  |\tilde K_{k+j} \ast \mu_t(x)| \le C  2^{dk}  2^{j(d-b)} (1+  2^k|x| )^{-N} .\Ee
Indeed,   note that
\begin{align*}
  \tilde K_{k+j} \ast \mu_t(x)&=  2^{d(k+j)} \int (1+ 2^{k+j}|x- ty |)^{-N} d\mu(y)
   \\ 
  & \le C2^{d(k+j)} \int (1+ 2^{j}| t^{-1}x- y|)^{-N} d\mu(y) 
  \end{align*}
  for $t\in [2^{-k}, 2^{-k+1}] $. 
  By Lemma \ref{conv}, we have
  \[\tilde K_{k+j} \ast \mu_t(x)   \lesssim  2^{d k}  2^{j(d-b)} .\]
 Also, note $| t^{-1}x- y|> 2^{-1} |t^{-1} x|$ if  $ |x|> 2^{-k+2}$. So, we have 
  \begin{align*}
  \tilde K_{k+j} \ast \mu_t(x)  \lesssim 
     2^{d(k+j)}  (1+ 2^{k+j}|x|)^{-N}   
   \end{align*} 
  provided that $ |x|> 2^{-k+2}$.  Combining the above two inequalities, we obtain \eqref{kast}.

Since the kernel of $P_{\le k}$  is bounded by a constant   
times  $ \tilde K_{k}$. We have   $|P_{\le k}f\ast \mu_t|  \le C |f|\ast \tilde K_{k} \ast \mu_t$.  By combining this and  \eqref{kast} it follows that 
\[  |P_{\le k}f\ast \mu_t|  \le C |f|\ast \tilde K_{k}\]
for $t\in [2^{-k}, 2^{-k+1}] $.  Recalling \eqref{m0} and taking supremum over $k$, we note that
\[  M_0 f \le CM_{H\!L} f,\] 
where $M_{H\!L}$ is the Hardy--Littlewood maximal operator. 
Thus, $M_0$ is bounded on $L^p$ for $1<p\le \infty$.  Hence,  to prove Theorem \ref{decay+dim}, we only need to show 
\Be \label{rweak} \Big\|\sum_{j\ge 1}^\infty M_j  f\Big\|_{L^{p_{(a,b)},\infty}(\R^d)} \le \|f\|_{L^{p_{(a,b)},1}(\R^d)}.  \Ee
which, of course, gives \eqref{rsweak}. Interpolation between \eqref{rsweak} and the trivial $L^\infty$ estimate gives \eqref{maximal} for 
$p>p_{(a,b)}$.

Similarly as before, thanks to  Lemma \ref{B-Sum-Tri},   to show the restricted weak-type \eqref{rweak} it is sufficient to obtain  
\Be \label{int}  \| M_j f\|_{L^p(\R^d)}\le C 2^{j (2(d-b) + 2a -1)(\frac1p- \frac{d-b+2a-1}{ 2(d-b) + 2a -1})} \|f\|_{L^p(\R^d)}   \Ee
for $1<p\le 2$. Recall \eqref{pab} and note that the exponent becomes zero when $p=p_{(a, b)}$. 
This in turn follows from interpolation between the two estimates: 
\begin{align} 
\label{mjL1}  \|M_j f\|_{L^{1,\infty}(\R^d)} & \le C 2^{j(d-a)} \|f\|_{L^1(\R^d)}, 
\\ 
\label{mjL2}  \|M_j f\|_{L^{2}(\R^d)} & \le C 2^{j(\frac12-a)} \|f\|_{L^2(\R^d)}.
\end{align}

The first inequality \eqref{mjL1} is easy to verify. 
Since the kernel of $P_{k+j}$ is bounded by a constant   
times  $ \tilde K_{k+j}$, we have $|P_{k+j}f\ast \mu_t|  \le C |f|\ast \tilde K_{k+j} \ast \mu_t$. Thus, using the inequality \eqref{kast}, 
 we have 
\[    |P_{k+j}f\ast \mu_t|  \le C 2^{j(d-a)} \tilde K_k\ast |f|(x)  \]  
for $t\in [2^{-k}, 2^{-k+1}] $. 
Therefore, from \eqref{mj}  it follows that
\[ M_j f(x)  \le C 2^{j(d-a)} M_{H\!L} \ast |f|(x)  .\]
Consequently, by the weak $L^1$ bound for the Hardy--Littlewood maximal function we obtain \eqref{mjL1}.


\subsection{Proof of  $L^2$ estimate \eqref{mjL2}} 
For the purpose of proving \eqref{mjL2},  we make use of the next lemma.

\begin{lemma} 
\label{loc-glob} Let $2\le p< \infty$,  $j>0$. Suppose that 
\Be\label{local-assum} \|M_\mu^{loc} P_j f\|_{L^p(\R^d)}\le C 2^{\eta j} \|f\|_{L^p(\R^d)}\Ee
for some $\eta\in \mathbb R$. Then, we have 
\[ \|M_j f\|_{L^p(\R^d)}\le C 2^{\eta j} \|f\|_{L^p(\R^d)}.\] 
\end{lemma}

\begin{proof}  By scaling it is easy to see that 
\[   P_{k+j}f\ast \mu_t(x)=   P_{j}(f)_{2^{-k}}\ast \mu_{2^k t}(2^k x) .   \]
Thus, it follows that 
\Be
\label{kkk}
 \sup_{2^{-k}\le t< 2^{-k+1}}   |P_{k+j}f\ast \mu_t(x)| =  
M_\mu^{loc} P_{j}(f)_{2^{-k}}(2^k x) .
\Ee

Recalling the definition of $M_j$ and using $\ell^p\hookrightarrow \ell^\infty$, we have 
\[  (M_j f)^p \le   \sum_k \Big(  \sup_{2^{-k}\le t< 2^{-k+1}}   |P_{k+j}f\ast \mu_t(x)| \Big)^p  .\]
Thus, by \eqref{kkk} we have  
\[  \|M_j f\|^p_{L^p(\R^d)} \le  \sum_{k=-\infty}^\infty  2^{-dk}    \| M_\mu^{loc} P_{j}(f)_{2^{-k}}\|_{L^p(\R^d)}^p.\]

Recall that $\beta\in C_c^\infty((1/2,2))$ is chosen in Section \ref{sec:endp}.
Let $\tilde \beta\in C_c^\infty((1/4, 4))$ such that $\beta=\tilde\beta \beta$. Define  
\[  \widehat{\tilde P_{j}f}(\xi) =  \tilde \beta(2^{-j}|\xi|)  \widehat f(\xi).\]
Using the assumption \eqref{local-assum}, we have
\begin{align*}  \|M_j f\|^p_{L^p(\R^d)} \le C   2^{\eta p j} 2^{-dk} \sum_{k=-\infty}^\infty   \| 
\tilde P_j (f)_{2^{-k}}\|_{L^p(\R^d)}^p 
\le C  2^{\eta p j} \sum_{k=-\infty}^\infty   \| 
\tilde P_{j+k} f \|_{L^p(\R^d)}^p. 
\end{align*} 
 The desired inequality now follows 
since 
\[  \Big( \sum_{k=-\infty}^\infty   \| 
\tilde P_{k} f \|_{L^p(\R^d)}^p\Big)^\frac1p  \le C\|f\|_{L^p(\R^d)}\]
for $2\le p\le \infty$. This  inequality can be shown by interpolation between the estimates for $p=2$ and $p=\infty$. 
The case $p=2$ is a straightforward consequence of the Plancherel theorem and the case $p=\infty$ is clear since the kernel of  $\tilde P_{k}$ has a uniformly bounded $L^1$ norm. 
\end{proof}

Thanks to Lemma \ref{loc-glob}, to prove \eqref{mjL2}, we only need to show 
\Be\label{mjL2-local}   \|M_\mu^{loc} P_j f\|_{L^{2}(\R^d)} \le C 2^{j(\frac12-a)} \|f\|_{L^2(\R^d)}.\Ee

\begin{lemma}\label{decay-} Let   $\mu$  be  a Borel measure with $\supp \mu\subset B(0,1)$ and $\psi$ be a smooth function on $\mathbb R^d$.  Suppose  \eqref{decay} holds. Then, we have
\[ | \widehat {\psi\mu}(\xi)|\le C|\xi|^{-a} .\]
\end{lemma} 

\begin{proof}  Since $\psi$ is smooth on $[-\pi, \pi]^d$, expanding $\psi$ in Fourier series, we have 
\[ \psi(x)= \sum_{\mathbf n\in \mathbb Z^d}  C_{\mathbf n} e^{i\mathbf n\cdot x}\]
with $C_{\mathbf n}$ satisfying $|C_{\mathbf n}|\le C(1+|\mathbf n|)^{-M}$  
for any $M$.  Since $\supp \mu\subset B(0,1)$, 
  \[ \widehat {\psi\mu}(\xi)=  \int e^{-2 \pi i x \xi}  \psi(x)  d\mu(x)   =  \sum C_{\mathbf n}\, \widehat \mu(\xi-\frac {\mathbf n}{2\pi}).\]
Splitting the sum into  two cases  $|\mathbf n|\le |\xi|$  and $|\mathbf n|> |\xi|$, we  have 
\[  | \widehat {\psi\mu}(\xi)|\le \sum_{|\mathbf n|\le |\xi|} |C_{\mathbf n}| | \widehat \mu(\xi-\frac {\mathbf n}{2\pi})| +   \sum_{|\mathbf n|> |\xi  
|}  |C_{\mathbf n}|| \widehat \mu(\xi-\frac {\mathbf n}{2\pi})|.\] 
Using use \eqref{decay} and rapid decay of $|C_{\mathbf n}|$, we  obtain 
  \begin{align*} 
 | \widehat {\psi\mu}(\xi)| 
\le  \sum_{|\mathbf n|\le |\xi|} |C_{\mathbf n}|  |\xi|^{-a}+ C  \sum_{|\mathbf n|> |\xi  
|}  (1+|\mathbf n|)^{-N} 
\le   C  |\xi|^{-a} 
\end{align*} 
as desired. 
\end{proof} 

Let $F(x,t)$  be a function defined on $\mathbb R^d\times I$ that is  differentiable in $t$. 
Applying the fundamental theorem of calculus to  $|F(x, \cdot)|^r$ and integrating in $x$, we have 
\[ 
\| \sup_{t\in I}| F(x,t) | \|_{L^r(\mathbb R^d)} \lesssim  \|F \|_{L^r(\mathbb R^d \times I)} + \| F\|_{L^r(\mathbb R^d \times I)}^{(r-1)/r}\|\partial_t F \|_{L^r(\mathbb R^d \times I)}^{1/r} .
\]
Applying Young's inequality to 
\[    \| F\|_{L^r(\mathbb R^d \times I)}^{(r-1)/r}\|\partial_t F \|_{L^r(\mathbb R^d \times I)}^{1/r}=2^{j\frac{r-1}{r^2}}\| F\|_{L^r(\mathbb R^d \times I)}^{(r-1)/r} \times 2^{-j\frac{r-1}{r^2}}\|\partial_t F \|_{L^r(\mathbb R^d \times I)}^{1/r}\]  
with exponents $r/(r-1)$ and $r$, we have
\begin{equation}\label{ftc}  \| \sup_{t\in I}| F(x,t) | \|_{L^r(\mathbb R^d)} \lesssim 
2^{\frac j r} \| F\|_{L^r(\mathbb R^d \times I)} 
+    2^{ j(\frac 1r-1)}  \|\partial_t F \|_{L^r(\mathbb R^d \times I)}
\end{equation}
(see, for example, \cite{L}).

\begin{proof}[Proof of \eqref{mjL2-local}]  Applying  the inequality \eqref{ftc} to $F(x,t)= \mu_t\ast P_j f(x)$, we have 
\[ \|M_\mu^{loc} P_j f\|_{L^2(\R^d)} \lesssim   2^{\frac j 2}  \|P_j f\ast \mu_t \|_{L^2(\mathbb R^d \times I)}+  2^{- \frac j2}  \|   \partial_t (P_j f\ast \mu_t) \|_{L^2(\mathbb R^d \times I)}.\]

Define a measure $\mu_l$ by  setting 
\[   (g, \mu_l)=   - 2\pi i \int g(x)\,  x_l\, d\mu(x), \quad g\in  C_c( \mathbb  R^d) .\] 
Since $\widehat \mu_l(\xi)=  \partial_l  \widehat{\mu}$, we observe that 
\begin{align*}
 \partial_t  \big( f\ast (P_j\mu)_t \big)  
&  =   \int e^{2\pi i x \cdot \xi} \beta(2^{-j}\xi) \widehat f(\xi)  \sum_{l=1}^d \partial_{l}   \widehat \mu(t\xi)\xi_l\, d\xi \\
& = 2^j \sum_{l=1}^d \int e^{2\pi i x \cdot \xi} \beta (2^{-j}\xi) \widehat f_l(\xi) \widehat{ \mu_l} (t\xi) d\xi \\
& =2^j \sum_{l=1}^d  (P_j  f_l)\ast (\mu_l)_t ,
\end{align*}
where $\widehat{f_l}(\xi) = 2^{-j}\xi_l \widehat f(\xi)$.

Therefore,  by Plancherel's theorem, \eqref{mjL2-local} follows if we show 
\[ \|  \beta(2^{-j}|\cdot|)  \widehat f\, \widehat{\mu}(t\cdot)  \|_{L^2(\mathbb R^d \times I)} \lesssim 2^{- aj} \|f\|_{L^2(\R^d)}\] 
and 
 \[ \|  \beta(2^{-j}|\cdot|)  \widehat f_l\, \widehat{\mu_l}(t\,\cdot)  \|_{L^2(\mathbb R^d \times I)} \lesssim 2^{- aj} \|f\|_{L^2(\R^d)}.\] 
The first is a straightforward consequence of \eqref{decay}. For the latter, since 
$${\| \beta(2^{-j}|\cdot|)  \widehat f_l \|_{L^2(\R^d)}\lesssim \|f\|_{L^2(\R^d)}},$$ we only need to show 
\[  |\widehat{\mu_l}(\xi)|\le C|\xi|^{-a}.\]
This follows from Lemma \ref{decay-}. 
\end{proof}

\subsection{Sharpness of the range of $p$} \label{sharpness} 
 Sharpness of the range in Theorem \ref{decay+dim} for some specific cases  also can be shown by considering 
some examples.  For later use, we show a slightly stronger result that  proves failure $L^p$ boundedness of the 
local maximal function.

 For the purpose, we use the following elementary lemma.

\begin{lemma}
\label{lorentz}   Let $0<\beta \le d$ and $\gamma \le 1$.  Set 
\[  h(x) = |x|^{-\beta} (\log 1/|x| )^{-\gamma } \chi_{B(0,2^{-1/d})}(x) .\] 
Then, $h\in L^{\frac d\beta, r}(\R^d)$ if $r> \frac 1 \gamma$ but $h\not\in L^{\frac d\beta, \frac1\gamma}(\R^d)$. 
\end{lemma} 

\begin{proof}
On the other hand, we have 
\[
h= |x|^{-\beta} (\log 1/|x| )^{-\gamma} \chi_{B(0,2^{-1/d})} \sim  \sum_{k=1}^{\infty} 2^{\frac{\beta}{d}k}  k^{-\gamma} \chi_{\{ y:  2^{-(k+1)/d} \le  |y| <  2^{-k/d}\} }(x).
\]
Then, from the definition of Lorentz spaces we have 
\[ \| h\|_{L^{p,r}(\R^d)} \sim \| \{ 2^{\frac{\beta}{d}k}  k^{-\gamma} 2^{-k/p} \}_{k=1}^\infty \|_{\ell_k^r}\] 
(see, for example, \cite[Remark 6.7]{T2}). 
In particular, with $p=\frac d\beta$, 
this gives 
$\| h\|_{L^{\frac d\beta,r}(\R^d)}  \sim \| \{ k^{-\gamma}\}_{k=1}^\infty \|_{\ell^r} < \infty$ if and only if  $r >1/\gamma$. 
\end{proof}

In particular, let us consider 
the measure given by 
\[    d\mu_\alpha(x) = \frac{1}{\Gamma(\alpha)}
 \bigl(1 - |x|^2\bigr)^{\alpha-1}_+ dx,   \quad \alpha \in (0, 1). \] 
 Via analytic continuation, the measure  $ \mu_\alpha$ coincides with the surface measure on the unit sphere when $\alpha=0$.   
The  Fourier transform of $\mu_\alpha$ is given by 
\[   \widehat{\mu_\alpha}(\xi)=  \pi^{-\alpha+1}
|\xi|^{-d/2-\alpha+1}
J_{d/2+\alpha-1}(2\pi |\xi|), \]
where $J_\beta$ denotes the Bessel function of order $\beta$ (see \cite[p.~171]{Stein1970}).   Thus, $\mu_\alpha$ satisfies  \eqref{decay} with $a =(d-1)/2+ \alpha$. A simple computation shows  that   $\mu_\alpha$ also satisfies \eqref{dim} with $b = d-1+\alpha$.   Thus, we have 
\[ p_{(a,b)} = d/(d-1+\alpha).\] 

It suffices to show that \eqref{pqmax} fails for $p \le p_{(a,b)}$.
Failure of the maximal estimate for $p< d/(d-1+\alpha)$ is easy to show.  Considering $f=\chi_{B(0, \delta)}$, one can easily  see 
\[M_{\mu_\alpha}^{loc}  f(x) \gtrsim \delta^{d-1+\alpha}, \quad  1\le |x|\le 2.\]
Thus, the maximal inequality \eqref{pqmax} implies 
$ \delta^{d-1+\alpha}\le C \delta^\frac dp.$ Letting 
$\delta\to 0$ gives the condition $p\ge d/(d-1+\alpha)$. 
Thus, if $p< p_{(a,b)}:=d/(d-1+\alpha)$,  \eqref{pqmax} fails.

 To show failure of the maximal bound for 
$p= d/(d-1+\alpha),$  we consider 
\[  f(x)
= |x|^{-(d-1+\alpha)} \Big (\log \frac{1}{|x|}\Big)^{-1} \chi_{B(0,1/2)}(x).   \] 
We claim that 
\Be
\label{infty}  f\ast (\mu_\alpha)_{|x|} (x)  =\infty  
\Ee 
for $1\le |x|\le 2$. By Lemma \ref{lorentz}, $f\in L^{\frac{d}{d-a+1},r}(\R^d)$ for any $r>1.$ Thus,  \eqref{infty}   shows 
that  $M_\mu $
cannot be bounded from $L^{\frac{d}{d-a+1},r}(\R^d)$  to 
$L^{\frac{d}{d-a+1},\infty}(\R^d)$ for any $r>1$.

Since both $f$ and $\mu_\alpha$ are radial, so is  $f\ast \mu_\alpha$.  Thus, $f\ast (\mu_\alpha)_{|x|} (x)= f\ast (\mu_\alpha)_{|x|} (|x|e_1)$.  Using this, we note that 
\begin{align*}  f\ast (\mu_\alpha)_{|x|} (x) 
&= \frac{1}{\Gamma(\alpha)}  \int f(|x|(1-y_1, \bar y))    \bigl(1 - y_1^2- |\bar y|^2\bigr)^{\alpha-1}_+  dy,  
\end{align*}
where  we write 
$y=(y_1, \bar y)\in \mathbb R\times \mathbb R^{d-1}$. 
Changing variables $s=1-y_1$ and then  $\tau=s(2-s)$   gives 
\begin{align*}
f\ast (\mu_\alpha)_{|x|} (|x|e_1)  
&\gtrsim  \int_0^{1/2}\int f(|x|(s, \bar y))    \bigl(  s(2-s)- |\bar y|^2\bigr)^{\alpha-1}_+  d\bar y\,ds
\\ 
&\gtrsim   \int_0^{3/4} \int_{s(\tau)\le |\bar y|}  f(|x|(s(\tau), \bar y))    \bigl(  \tau - |\bar y|^2\bigr)^{\alpha-1}_+  d\bar y\,d\tau.
\end{align*}
Since  $ s\le\tau \le 3s/2$ and $1\le |x|\le 2$, it follow that 
\begin{align*}
f\ast (\mu_\alpha)_{|x|} (|x|e_1)   &\gtrsim   \int_0^{3/4}\int_{\tau \le |\bar y|}  f(|x|(s(\tau), \bar y))    \bigl(  \tau - |\bar y|^2\bigr)^{\alpha-1}_+  d\bar y \, d\tau . 
\end{align*}  
Moreover, we have $ f(|x|(s(\tau), \bar y)) \sim   |\bar y|^{-(d-1+\alpha)} (\log \frac{1}{|\bar y|})^{-1} $ provided that 
$ f(|x|(s(\tau), \bar y))\neq 0$.  Thus, for a small constant $c>0$, we have 
\[f\ast (\mu_\alpha)_{|x|} (|x|e_1)   \gtrsim  \int_{|\bar y|< c} \int_{0<\tau \le |\bar y|}   |\bar y|^{-(d-1+\alpha)} (\log \frac{1}{|\bar y|})^{-1}   \bigl(  \tau - |\bar y|^2\bigr)^{\alpha-1}_+  d\tau\, d\bar y.\]

Note that $\int_0^{|\bar y|}    \bigl(  \tau - |\bar y|^2\bigr)^{\alpha-1}_+  d\tau\sim |\bar y|^\alpha$.  Consequently,  it follows that 
\begin{align*}
f \ast (\mu_\alpha)_{|x|} (x) & \gtrsim  \int_{\mathbb R^{d-1}} |\bar y|^{-d+1} (\log \frac{1}{|\bar y|})^{-1} \chi_{B(0, c)}(\bar y) d \bar y 
\end{align*}
for $1\le |x| \le 2$. The right hand is clearly infinite. Thus, we  conclude \eqref{infty}.

\subsection{Failure of the spherical maximal bound on   $L^{\frac d{d-1}, r}(\R^d)$} \label{sss} 
In this section we show failure of the maximal bound for the local spherical maximal function for the critical exponent $p=\frac d{d-1}$.

\begin{proposition} 
\label{lorentz-s} Let $\sigma$ be the normalized surface measure on $\mathbb S^{d-1}$ and 
\[M_{sph}^{loc}f (x ) = \sup_{t\in I} |f \ast \sigma_t(x)|.\] 
Then 
\Be
\label{s-endp}
 \| M_{sph}^{loc} f \|_{L^{\frac{d}{d-1},\infty}(\R^d)} \lesssim \| f\|_{L^{\frac{d}{d-1},r}(\R^d)}
 \Ee
  holds only if $r \le  1$.  
\end{proposition}

\begin{proof}
Let us consider 
\[
f(x) = |x|^{-d+1} (\log \frac{1}{|x|})^{-1} \chi_{B(0,1/2)}(x).
\]
Since $f\in L^{\frac{d}{d-1},r}$ for all $r>1$ by Lemma \ref{lorentz}, in order to prove  failure of the estimate \eqref{s-endp}, it is sufficient to show that 
\Be
\label{maxi}
M_{sph}^{loc} f (x) = \infty 
\Ee
for $  1\le  |x| \le 2$. 

As before,   $f\ast \sigma$ is radial.   Thus,    
\[ f\ast \sigma_{|x|} (x)=f\ast \sigma_{|x|} (|x|e_1)=  \int_{\mathbb S^{d-1}} f(|x|(1-y_1, \bar y)) d\sigma(y).\] 
We parametrize the sphere near $e_1$, i.e., $y_1=\psi(\bar y):=\sqrt{1- |\bar y|^2}$, and  note that $1-\psi(\bar y) \le  |\bar y|$. Since $|x|\sim 1$,   we have 
\begin{align*}
f \ast \sigma_{|x|} (x) & \gtrsim  \int_{\mathbb R^{d-1}} |\bar y|^{-d+1} (\log \frac{1}{|\bar y|})^{-1} \chi_{B(0,1/2)}(\bar y) d \bar y.
\end{align*}
  Note that the right hand side is infinite.  Thus, \eqref{maxi} follows. 
\end{proof}

\section{$L^p$-improving estimates for the local maximal operator}
\label{sec4}
In this section, we consider $L^p$--$L^q$ bound for the local maximal operator $M_\mu^{loc}$.  

\subsection{Estimates relying on \eqref{decay} and \eqref{dim}} 
The estimates used for the proof of  Theorem \ref{decay+dim}  can be combined with  $L^1$--$L^\infty$ estimate for frequency localized maximal average, 
to prove  the estimate  \eqref{pqmax}  for some $(p,q)$ satisfying  $q >p$ and $1/q \ge 1-1/p$.  

For given $a$ and $b$,  let $\Delta  \subset[0,1]^2$ denote the  closed triangle  with vertices 
\Be
\label{vertices} O=(0,0), \  P=(1/p_\circ,1/p_\circ),  \   Q = (1/p_{(a,b)},1/p_{(a,b)}'),
\Ee
where $p_\circ = \min \{ p_a ,p_{(a,b)} \}$. By $[P,Q]$ we also  denote the closed line segment in the plane connecting $P, Q$. As mentioned in the introduction, we have the following theorem, which is largely depending on $b$ in contrast with Theorem \ref{RF}.

\begin{theorem}\label{thm:pq}
Let $\mu$ be a compactly supported Borel measure. Suppose that \eqref{decay} and \eqref{dim} hold.
Then, we have  the following. 

\begin{enumerate}[leftmargin=17pt, itemsep=2pt 
]
 \item[$(1)$] If $b\ge d-1$, then  \eqref{pqmax} holds for $(1/p,1/q) \in \Delta \setminus [P,Q]$.

 \item[$(2)$]
If $b< d-1$, then
\eqref{pqmax} holds for $(1/p,1/q) \in \Delta \setminus \{ P,Q \} $. 
\end{enumerate} 
Moreover, for $(1/p,1/ q) = P,Q$, we have  the restricted weak-type estimate  
\begin{equation}\label{pqrw}
\| M_\mu^{loc} f \|_{L^{q,\infty}(\R^d)} \le C \| f\|_{L^{p,1}(\R^d)}.
\end{equation}
\end{theorem}

\begin{proof}
Clearly,  Theorem \ref{endp} and  Theorem \ref{decay+dim} give restricted weak-type $(p,p)$ bound on  $ M_\mu^{loc} $ for  $p=p_{(a,b)}, p_a$.  
Thus, to prove  Theorem \ref{thm:pq}  we need only  to show  the restricted weak-type  $(p,q)$ estimate  \eqref{pqrw} for  $(1/p,1/ q) = Q$.  
Indeed, via real interpolation we obtain the desired conclusions ${(1)}$ and ${(2)}$ (see, for example, \cite{Stein1970}).   
It should be noticed that,  when $b\ge d-1$ (the case ${(1)}$),  interpolation of the estimates \eqref{pqrw} for $(1/p,1/ q) = P,Q$ does not give 
the strong bound \eqref{pqmax} for  $(1/p,1/ q)$ in the open line segment $(P,Q)$ since  $p_\circ=p_{(a,b)}$.

Similarly as  in the proof of Theorem \ref{decay+dim}, we have 
\Be
 \label{locj}  M_\mu^{loc} f \le \sum_{j=0}^\infty M_j^{loc} f,
\Ee
where
\begin{align*}
   M_0^{loc} f =    \sup_{t\in I}  |P_{\le 0} f\ast \mu_t|,    \quad 
  M_j^{loc} f=\sup_{t\in I}    |P_{j}f\ast \mu_t|, \quad j\ge 1
    \end{align*}
(cf. \eqref{m0} and \eqref{mj} with $k=0$).

To show  \eqref{pqrw} for  $(1/p,1/ q) = Q$, by Lemma \ref{B-Sum-Tri} we only need to show 
\[   \| M_j^{loc} f\|_{L^{p'}(\R^d)}\le C 2^{j (2(d-b) + 2a -1)(\frac1p- \frac{d-b+2a-1}{ 2(d-b) + 2a -1})} \|f\|_{L^p(\R^d)}   \] 
for  $1\le p\le 2$.
The estimate in turn follows by interpolation between \eqref{mjL2-local} and  
\begin{equation}
\label{10} \| M_j^{loc} f  \|_{L^\infty(\R^d)}   \lesssim  2^{(d-b)j}  \|f\|_{ L^1(\R^d)} .
\end{equation}

This estimate is easy to show by applying   \eqref{kast}. Indeed,  note  $|P_j f \ast \mu_t| \le C |f|\ast (\tilde K_j \ast \mu_t)$ for $t\in [1,2]$. 
Since   $\|\tilde K_j \ast \mu_t\|_{L^\infty(\R^d)}  \lesssim  2^{(d-b)j}$ for  $t\in [1,2]$ by Lemma \ref{conv}),  \eqref{10} follows.
\end{proof}

 Recall that   $p_\circ = p_{(a,b)}$  if $b\ge  d-1$, and  $p_\circ = p_a$  if $b < d-1$. 
 Thus,  the set 
$\Delta$ is given by 
\[
\Delta =  \Big\{ \Big( \frac1p,\frac1q \Big) : ~ \frac{d+2a -b-1}{(d-b)q} \ge  \frac {1}{p}\ge \frac1q, \  \     \frac {1} {p} \le \frac{1}{p_{(a,b)}} 
 \Big\}
\] 
if $b \ge d-1$.  When $b <  d-1$,
\[
\Delta = \Big\{ \Big( \frac1p,\frac1q \Big) :~ \frac{d+2a -b-1}{(d-b)q} \ge  \frac {1}{p}\ge \frac1q, \ \  2a + \frac{d-b-1}{q} \ge \frac {d+2a-b} {p}  \Big\} .
\]

Regarding  sharpness of Theorem \ref{thm:pq}, we can show that \eqref{pqmax} generally fails to hold if the last inequality appearing in each case is not satisfied. 
The inequalities determine the border line segment $(P, Q)$.

For the case $b  \ge d-1$, the condition $1/ {p} < 1/p_{(a,b)}$ is necessary for \eqref{pqmax}  to hold   under the assumptions    \eqref{decay} and \eqref{dim} as discussed in Section \ref{sharpness}.  
In the case of $b < d-1$,  the condition 
\Be 
\label{abd}
2a + \frac{d-b-1}{q} \ge \frac {d+2a-b} {p}
\Ee
also seems to be necessary for  \eqref{pqmax}.  
However, we can verify  this only for  $2a = b= k$, where  $k$ is an integer satisfying $2\le k \le d-1$.

Let $\mu$ be a smooth induced Lebesgue measure  on $k$-dimensional submanifold 
\Be
\label{Psi} \mathfrak S = \{ (u,\Psi(u)) \in \mathbb R^k\times \mathbb R^{d-k}: |u|\le c\}
\Ee  
for a sufficiently small $c>0$, where $\Psi$ is a smooth function in a neighborhood of the origin and $\Psi(0)\neq 0$.
Then $\mu$ becomes a Salem measure satisfying \eqref{decay} and \eqref{dim} with $2a = b=k$ if $\Psi$ satisfies  a nondegeneracy condition so-called \emph{strong curvature condition} on  $\mathfrak S$ such that, for every $\mathbf t \in \mathbb S^{d-1}$, 
\begin{equation}\label{strongcurvature}
| \det D^2\langle \mathbf t, \Psi ( u)\rangle | \ge \epsilon
\end{equation}
for all $| u|\le c$.   This type of nondegeneracy condition was assumed in \cite{SY, Sr} for counting rational points near/on submanifolds of codimension bigger than  one.  
See \cite{Ch82, Ba02} for earlier results in study of Fourier restriction under the same assumption.

In this case, \eqref{abd} equals 
\[ k + \frac{d-k-1}{q} \ge \frac {d} {p},
\]
which is necessary for \eqref{pqmax} to hold.  To show this, taking  $f= \chi_{B(0,\delta)}$, we get $f\ast  \mu_t(x) \gtrsim \delta^k$ if $x$ is contained in a $\delta$-neighborhood  $E(\delta)$  of $\cup_{t\in I} E_t$, where $E_t = \{t(u, \Psi(u)) : |u|  \le c \}$ for a sufficiently small constant $c>0$.  Note that $|E(\delta)|\gtrsim \delta^{d-k-1}$. Thus, 
\[ \delta^{k + \frac{d-k-1}{q}}\lesssim   \| \sup_{t\in I}  |f\ast  \mu_t|  \|_{L^q(\R^d)} .\]
 Since $\| f\|_{L^p(\R^d)} \lesssim \delta^{\frac d p}$, \eqref{pqmax} implies  $ \delta^{k + \frac{d-k-1}{q}}\lesssim    \delta^{\frac d p}$.
Letting $\delta \to 0$, we get the desired inequality.

\subsection{Exension via a Strichartz type  estimate} 
\label{ext}
We now discuss possible enlargement of the $p,q$ range in Theorem \ref{thm:pq} under an additional assumption.  For the purpose, we make use of  a Strichartz type estimate for the averaging operator $P_j f\ast \mu_t$.    In order to do so, in addition to \eqref{decay} and \eqref{dim}, we further assume  that  the measure $\mu$ satisfies 
a form of dispersive estimate 
\begin{equation}\label{dispersive}
\| P_j(\mu_t\ast  \overline \mu_s)\|_{L^\infty(\R^d)} \lesssim 2^{(d-2a) j}  (1+ 2^j|t-s|)^{-\delta}, \quad t, s\in  I
\end{equation}
for all $j\ge 1$ and some $\delta>0$,  while  $a, b,$ and $\delta$ satisfy 
\Be
\label{adelta-con}
2 a+\delta> b.
\Ee

It is easy to see that \eqref{decay} implies the estimate  \eqref{dispersive} when
$2^j |t-s| \le 1$. Moreover, one can readily verify that the estimate
\eqref{dispersive} does not depend on the particular choice of the cutoff
function $\beta$ defining the projection operator $P_j$.

The assumption \eqref{dispersive} is closely related to dispersive estimates
for evolution operators such as the wave and Schr\"odinger operators.
Such dispersive estimates play a crucial role in determining the admissible
pairs $(q,r)$ of exponents for the mixed-norm spaces $L_t^q L_x^r$ into which
the propagators map the initial data. In particular, the decay rate reflects
the curvature properties of the characteristic surfaces associated with the
underlying evolution operators (see, for example, \cite{KT, T}).

However, for a general measure $\mu$, the averaging operator
$f \mapsto P_j \mu_t * f$ cannot, in general, be expressed in the form of an
evolution operator. Moreover,  little is known about which fractal
measures satisfy \eqref{dispersive} for given constants $\alpha$ and $\delta$.
The most viable case appears to be when $2a \le d$, where geometric and
combinatorial arguments are more likely to yield such estimates for fractal
measures with specific structural properties, without using the Fourier transform.
 
Let us set 
\[ 
R=\begin{cases}  R_1:=   \Big(
\frac{(d-b+2a)(\delta+1)-d}{2(d-b+a)(\delta+1)-d}, \,
\frac{\delta(d-b)}{2(d-b+a)(\delta+1)-d}
\Big)  & \text{ if }   2a(\delta+1) >d,  
\\   
R_2:=    \Big(
\frac12,\,
\frac{d-2a-\delta}{2(d-\delta-1)}
\Big)    & \text{ if }    2a(\delta+1) \le d.  
\end{cases} 
\]

One can easily  verify that \eqref{adelta-con} 
ensures that the points
$R_1$ and $R_2$ lie outside~$\Delta$.

\begin{theorem}\label{pqr}
Let $\mu$ be a compactly supported Borel measure satisfying \eqref{decay} and \eqref{dim} for some $a>1/2$ and $ 0<b\le d$. 
Suppose that  \eqref{dispersive} holds with   $\delta$ satisfying  \eqref{adelta-con}. Then, 
 \eqref{pqmax} holds for $(1/p,1/q)$ contained in the interior of the convex hull of the points $O$, $P$, $Q$, and  $R$.  
Moreover,  if $2a(\delta+1) \neq d$,   \eqref{pqrw} holds for  $(1/ p_1, 1/q_1) = R$.
%
\end{theorem}

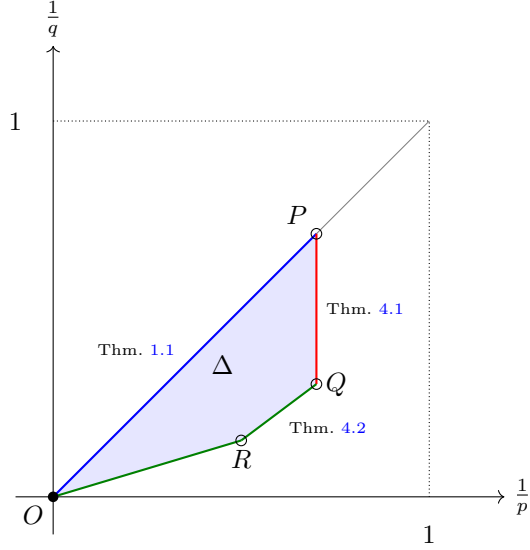
\begin{figure}[ht]
\centering
\begin{tikzpicture}[scale=5] 
    \draw[->] (-0.1,0) -- (1.2,0) node[right] {$\frac{1}{p}$};
    \draw[->] (0,-0.1) -- (0,1.2) node[above] {$\frac{1}{q}$};
    
    \draw[densely dotted] (0,1) -- (1,1) -- (1,0);
    \node at (1,-0.1) {$1$};
    \node at (-0.1,1) {$1$};
    
    \draw[gray, thin] (0,0) -- (1,1);
    
    \coordinate (O) at (0,0);
    \coordinate (P) at (0.7, 0.7);  
    \coordinate (Q) at (0.7, 0.3);  
    \coordinate (R) at (0.5, 0.15); 
    
    \fill[blue!10] (O) -- (P) -- (Q) -- (R) -- cycle;
    
    \draw[blue, thick] (O) -- (P) node[midway, above left, black, font=\tiny] {Thm. \ref{RF}}; 
    \draw[red, thick] (P) -- (Q) node[midway, right, black, font=\tiny] {Thm. \ref{thm:pq}};
    \draw[ao, thick] (Q) -- (R) node[midway, below right, black, font=\tiny] {Thm. \ref{pqr}};
    \draw[ao, thick] (R) -- (O);
    
    \node[above left] at (P) {$P$};
    \node[right] at (Q) {$Q$};
    \node[below] at (R) {$R$};
    \node[below left] at (O) {$O$};
    
    \fill (O) circle (0.4pt);
    \draw (P) circle (0.4pt);
    \draw (Q) circle (0.4pt);
    \draw (R) circle (0.4pt);
    
    \node at (0.45, 0.35) {$\Delta$};
\end{tikzpicture}
\caption{The $L^p$-improving region $\Delta$ described in Theorem \ref{pqr}, formed by the convex hull of $O, P, Q, R$.}
\label{fig:region}
\end{figure}

Clearly, as in Theorem \ref{thm:pq},  
some borderline cases of the convex hull can be included but we omit them for simplicity of the statement. 
We discuss below how Theorem \ref{pqr} applies to a few specific cases.   

We first recall the spherical measure $\sigma$ in $\mathbb R^d$ for $d\ge 3$. 
In this case,  we have \eqref{dispersive} with  $a = \delta = \frac{d-1}{2}$. This can be shown using the asymptotic expansion of the Fourier transform $\widehat{\sigma}$ of the spherical measure and Fourier decay of the surface measure on the cone $(\xi,|\xi|)$.  
So, the inequality $2a(\delta+1) > d$ holds. Hence, we have $R=R_1 = ( \frac{d(d-1)}{d^2+1}, \frac{d-1}{d^2+1} )$, which coincides with one of the endpoints of $L^p$-improving range for the spherical maximal function \cite{SS, L}.      More precisely, Schlag \cite{Schlag} (for $d=2$) and Schlag--Sogge \cite{SS} (for $d\ge 3$) showed that $M_{sph}^{loc}$ maps $L^p$ boundedly into $L^q$ for $(\tfrac1p,\tfrac1q)$ in the interior of the convex hull $\mathcal H$ of the points
\[
\mathbf O=(0,0),\quad 
\mathbf P_{2}=\Bigl(\tfrac{d-1}{d},\tfrac{d-1}{d}\Bigr),\quad 
\mathbf P_3=\Bigl(\tfrac{d-1}{d},\tfrac{1}{d}\Bigr),\quad 
\mathbf P_4=\Bigl(\tfrac{d(d-1)}{d^2+1}, \tfrac{d-1}{d^2+1}\Bigr),
\]
which correspond to the points $P$, $Q$, and $R$ in Theorem \ref{pqr}, respectively. They also showed that $M_{sph}^{loc}$ fails to be bounded from $L^p$ to $L^q$ whenever $(\tfrac1p,\tfrac1q)\notin \mathcal H$.

For the borderline cases, when $d=2$ (in which case $\mathbf P_2=\mathbf P_3$), the third author \cite[Theorem~1.1]{L} almost completely characterized the $L^p$--$L^q$ boundedness of $M_{sph}^{loc}$, except at $(\tfrac1p,\tfrac1q)=\mathbf P_4$. At this endpoint, a restricted weak-type $(p,q)$ estimate was established, while the corresponding strong-type bound remains open. When $d\ge 3$, restricted weak-type $(p,q)$ bounds at $(\tfrac1p,\tfrac1q)=\mathbf P_3$ and $\mathbf P_4$ were obtained in the same paper \cite{L} by adapting a similar strategy.

Taking advantage of this opportunity, we correct the statement in \cite[Theorem~1.4]{L} in which the pairs $(p, q)$ satisfying $(\tfrac1p,\tfrac1q)\in (\mathbf P_2,  \mathbf P_3)$ should be excluded 
from the strong boundedness range. Indeed,  by combining the restricted weak-type bounds at $(\tfrac1p,\tfrac1q)=\mathbf P_3$ and $\mathbf P_4$ with Bourgain's restricted weak-type $(\tfrac d{d-1},\tfrac d{d-1})$ estimate (see, for example, Theorem~\ref{endp}), and applying interpolation, we obtain the following result.

\begin{theorem}\label{leelee}
Let $d\ge 3$. There exists a constant $C>0$ such that
\[
\|M_{sph}^{loc} f\|_{L^q(\mathbb{R}^d)} \le C \|f\|_{L^p(\mathbb{R}^d)}
\]
for $(1/p,1/q)\in \mathcal H\setminus \bigl([\mathbf P_2,\mathbf P_3]\cup \{\mathbf P_4\}\bigr)$.
Moreover, for $(1/p,1/q)\in [\mathbf P_2,\mathbf P_3]\cup \{\mathbf P_4\}$, we have
\[
\|M_{sph}^{loc} f\|_{L^{q,\infty}(\mathbb{R}^d)} \le C \|f\|_{L^{p,1}(\mathbb{R}^d)}.
\]
\end{theorem}

Proposition~\ref{lorentz-s} shows that strong $L^p$--$L^q$ bounds fail, and that $L^{p,1}$ cannot be replaced by any larger Lorentz space $L^{p,r}$ with $r>1$, when $(1/p,1/q)\in [\mathbf P_2,\mathbf P_3]$. As noted in the proof of Theorem~\ref{thm:pq}, real interpolation does not yield the corresponding strong-type estimate.

In \cite{HK}, the  authors obtained maximal estimates associated with surfaces of half the ambient dimension in even dimensions $d=2n$ under the assumption \eqref{strongcurvature}.  For example, if $\mu$ is a surface measure on the complex curve $(x,y,x^2-y^2,2xy)$, which is a  2-dimensional surface in $\mathbb R^4$, then \eqref{dispersive} is satisfied with $a =\delta = 1$ and $d=4$. In this case,  $2a(\delta+1) \le 4$. Thus, we have  $R_2 = ( \frac12,  \frac14)$.   
For $n\ge 4$, we have \eqref{dispersive} with $a= \delta =n/2 $ in  $b = n$. Consequently, noting  $2a(\delta+1) >d$ gives  $R_1 = (\frac{2n}{3n+2}, \frac{n}{3n+2})$, which coincides with one of the vertices of the type set in   \cite[Theorem 1.2]{HK}.

 In order to prove Theorem \ref{pqr}, we  make use of the following Strichartz type estimates for the frequency localized average  $P_j f\ast \mu_t$, which we obtain by the standard $TT^\ast$ argument.

\begin{proposition}\label{prop:stri}
Suppose that   $\mu$ satisfies \eqref{decay} and \eqref{dispersive}. Then, we have the estimate 
\begin{equation}\label{eq:stri}
\| P_j f\ast \mu_t \|_{L^{\frac{2(\delta+1)}{\delta}}_{x,t}(\mathbb R^d\times I)} \le C 2^{ \frac{d-\delta-2a(\delta+1)}{2(\delta+1)} j} \|f\|_{L^2(\R^d)}.
\end{equation}
\end{proposition}

\begin{proof}
Let us set 
\[
E_j f(x,t) = \int e^{i x\cdot \xi} \beta(2^{-j}|\xi|) \widehat{\mu}(t\xi) f(\xi) d\xi, 
\]
so we have $E_j f (x,t) = P_j f^\vee \ast \mu_t$.  By Plancherel's theorem and duality, \eqref{eq:stri} is equivalent to 
\Be\label{TT*}   
\| E_j^\ast  g\|_{L^2(\R^d)} \le C 2^{ \frac{d-\delta-2a(\delta+1)}{2(\delta+1)} j} \|g\|_{L^{\frac{2(\delta+1)}{\delta+2}}_{x,t}(\mathbb R^d\times I)}. 
\Ee
Here $E_j^\ast$ denotes the adjoint operator of $E_j$, which is given by
\[E_j^\ast g(\xi) =  \beta(2^{-j}|\xi|)  \iint_{\mathbb R^d\times I} e^{-i y\cdot \xi}\,  \overline{\widehat\mu}(s\xi)  g(y,s) dy ds .\]

Thus, we have 
\begin{align*}
E_j E_j^\ast g(x,t) & = \iint_{\mathbb R^d\times I}  \int e^{i (x-y)\cdot \xi} \beta^2(2^{-j}|\xi|) \widehat \mu(t\xi) \overline{\widehat\mu}(s\xi) d\xi g(y,s) dy ds .
\end{align*}
Since $\| E_j^\ast  g\|_{L^2(\R^d)}^2 = \langle   E_j E_j^\ast  g, g \rangle$, by duality  \eqref{TT*} follows from the estimate 
\begin{equation}\label{dual}
\| E_j E_j^\ast g\|_{L^{\frac{2(\delta+1)}{\delta}}_{x,t}(\mathbb R^d\times I)} \le 2^{ \frac{d-\delta-2a(\delta+1)}{ \delta+1 } j} \| g\|_{L^{\frac{2(\delta+1)}{\delta+2}}_{x,t} (\R^d\times I)} .
\end{equation}

For each $t,s\in I$, we define  
\[
U(t,s) h (x) = \iint e^{i (x-y)\cdot \xi} \beta^2(2^{-j}|\xi|) \widehat \mu(t\xi) \overline{\widehat\mu}(s\xi) d\xi\, h(y) dy .
\]
Note $U(t,s) h(x) = K_j(t,s) \ast h(x)$ where
\[
K_j{(t,s)} (x) =  \int e^{i x\cdot \xi} \beta^2(2^{-j}|\xi|) \widehat \mu(t\xi) \overline{\widehat\mu}(s\xi) d\xi  .
\]
From \eqref{decay} note that  $ \| \widehat{K_j{(t,s)}}   \|_{L^\infty(\R^d)}  = \| \beta(2^{-j}|\xi|) \widehat \mu(t\xi) \overline{\widehat\mu}(s\xi) \|_{L^\infty(\R^d)} \lesssim {2^{-2aj}}$. 
Thus, it follows by Plancherel's theorem that
\[
\| {U(t,s) h } \|_{L^2(\R^d)}  \lesssim 2^{-2aj} \|h\|_{L^2(\R^d)}, \quad t, s\in I. 
\]
On the other hand, the assumption \eqref{dispersive} gives  a dispersive estimate
\[
\| {U(t,s) h} \|_{L^\infty(\R^d)}  \lesssim 2^{(d-2a-\delta)j} |t-s|^{-\delta} \|h\|_{L^1(\R^d)},  \quad t, s\in I. 
\]
Interpolating  these two estimates gives 
\begin{equation}
\label{interpol}
\| {U(t,s) h} \|_{L^q(\R^d)}  \lesssim 2^{(d-2a-\delta - (d-\delta)\frac2q)j} |t-s|^{-\delta(1-\frac 2q)} \|h\|_{L^{q'}(\R^d)}, \quad t, s\in I
\end{equation}
for $q\ge 2$. 

Noting that $ E_j E_j^\ast g= \int_I U(t,s)  g_s ds$, where $g_s= g(\cdot, s)$,  by \eqref{interpol} we have 
\begin{align*}
\| E_j E_j^\ast g\|_{L^q_{x,t}(\R^d\times I)} & \le \bigg\| \int_I \| U(t,s)  g_s(\cdot) \|_{L^q(\R^d)} ds \bigg\|_{L^q_t(I)} \\
 & \lesssim  2^{ ( (d-\delta)(1-\frac2q) - 2a  )j} \bigg\| \int_I |t-s|^{-\delta (1-\frac 2q)} \|g_s(\cdot)\|_{L^{q'}(\R^d)} ds \bigg\|_{L^q_t(I)}. 
 \end{align*} 
By Hardy--Littlewood--Sobolev inequality, we obtain 
\begin{align*}
\| E_j E_j^\ast g\|_{L^q_{x,t}(\R^d\times I)}   \lesssim  2^{ ( (d-\delta)(1-\frac2q) - 2a  )j} \| g \|_{L^{q'}_{x,t}(\R^d\times I)} 
\end{align*}
when $1/q = 1/q' - (1- \delta(1-\frac2q))$ i.e., $q = 2(\delta+1)/\delta$. 
Hence we get \eqref{dual}. 
\end{proof}

Once we have the estimate  \eqref{eq:stri},  Theorem \ref{pqr} can be proved in the similar manner as before.

\begin{proof}[Proof of Theorem \ref{pqr}] 
Recalling \eqref{locj},  it is enough to consider the frequency localized operators $M_j^{loc}$ for $j\ge 0$.  

Since we already have the estimate \eqref{pqrw} for  $P$ and $Q$ from Theorem   \ref{thm:pq},  
from the perspective of interpolation, it is sufficient to show  the restricted weak-type estimate \eqref{pqrw} for $(1/p,1/q) = R$ or  $(1/p,1/q)$ arbitrarily  close to $R$. To this end we consider 
the cases $ 2a(\delta+1) >d$, $2a(\delta+1) < d$, and $2a(\delta+1) = d$, separately.  

By Proposition \ref{prop:stri} and the inequality \eqref{ftc}, we have 
\begin{equation}\label{2q}
 \| M_j^{loc}   f  \|_{L^{\frac{2(\delta+1)}{\delta}}(\R^d)} \le C 2^{- \frac{2a(\delta+1)-d}{2(\delta+1)}  j} \| f\|_{L^{2}(\R^d)} .
\end{equation}  
 
If $ 2a(\delta+1) >d$,  we apply Lemma \ref{B-Sum-Tri}  to the estimates  \eqref{10} and  \eqref{2q}  with
taking $p_0=1, q_0=\infty$,  $\epsilon_0 = d-b$, $p_1=2, q_1=\frac{2(\delta+1)}{\delta}$, and $\epsilon_1 =  \frac{2a(\delta+1)-d}{2(\delta+1)} $. 
Consequently, we get  
\begin{equation}\label{rwmj}
\Big\|\sum_{j\ge 0} M_j^{loc}  f \big\|_{L^{q,\infty}(\R^d)} \lesssim \| f\|_{L^{p,1}(\R^d)}.
\end{equation}
for $(1/p,1/q) = R$, which gives \eqref{pqrw} for $(1/p,1/q) = R= R_1$.

If $ 2a(\delta+1) < d$,  instead of  the estimate \eqref{10}   we  apply Lemma \ref{B-Sum-Tri} to  \eqref{2q} and \eqref{mjL2-local}.
Thus,  taking $p_0=2, q_0=\frac{2(\delta+1)}{\delta},$  $\epsilon_0 =\frac{d-2a(\delta+1)}{2(\delta+1)}>0$,   $p_1=q_1=2$, and $\epsilon_1 = a- 1/2 $ in Lemma \ref{B-Sum-Tri},  we get \eqref{rwmj} for $(1/p,1/q) = R=R_2$.  
 
Finally, for the case $ 2a(\delta+1) = d$, we have $R_2 = (\frac12,\frac{\delta}{2(\delta+1)})$. Thus,  interpolating   \eqref{2q} and the trivial $L^\infty$ bound,   we have    
\[  \| M_j^{loc} f \|_{L^{q }(\R^d)} \le C  \| f\|_{L^{p}(\R^d)}, \quad (1/p,1/q) \in [O,R_2].\]
Further interpolation with \eqref{mjL2-local} yields 
$
\| M_j^{loc} f \|_{L^{ q }(\R^d)} \le C 2^{-\tau  j} \| f\|_{L^{ p}(\R^d)}
$
for some $\tau>0$, provided that $(1/p,1/q)$ is contained in an open triangular region  $\mathcal T$ with vertices $O$, $(1/2,1/2)$, and $R_2$. 
Recalling \eqref{locj},  by summation along $j$  we obtain \eqref{pqmax} for  $(1/p,1/q) \in\mathcal T$.  Additional interpolation with 
the already established estimates in Theorem  \ref{thm:pq} gives \eqref{pqmax} for  $(1/p,1/q)$ contained in the interior of the convex hull of  
$O$, $P$, $Q$, and  $R$.  \end{proof}

\bigskip

\section*{Acknowledgement} 
This work was supported by the National Research Foundation of Korea (NRF) through Grant Nos. RS-2025-24533387 (Ham), RS-2024-00342160 (Ham, Kah, Lee), and ARC DP 260100485 (Li).

\bigskip

\bibliographystyle{plain}

\end{document}